\documentclass[11pt,twoside, leqno]{article}

\usepackage{amsthm}
\usepackage{amssymb}
\usepackage{graphics,color}
\usepackage{amsmath}
\usepackage{mathrsfs}
\usepackage{txfonts}
\usepackage{graphics,color}
\usepackage[Symbol]{upgreek}

\allowdisplaybreaks 
\def\rr{{\mathbb R}}
\def\rn{{{\rr}^n}}

\def\cn{{\mathbb N}}

\def\supp{{\mathop\mathrm{\,supp\,}}}

\def\loc{{\mathop\mathrm{\,loc\,}}}

\def\ez{\epsilon}

\def\gz{{\gamma}}

\def\vz{\varphi}

\def\hs{\hspace{0.3cm}}

\def\div{\mathrm{div}}
\def\Exp{\mathrm{Exp}}

\def\Exp{\mathrm{Exp}}

\newtheorem{thm}{Theorem}[section]
\newtheorem{lem}[thm]{Lemma}
\newtheorem{prop}[thm]{Proposition}
\newtheorem{rem}[thm]{Remark}

\newtheorem{defn}[thm]{Definition}

\numberwithin{equation}{section}

\begin{document}
\arraycolsep=1pt
\author{} \arraycolsep=1pt
\arraycolsep=1pt

\title{\bf\Large Flows for non-smooth vector fields \\with subexponentially integrable
divergence
\footnotetext{\hspace{-0.35cm} 2010 \emph{Mathematics Subject
Classification}. Primary 35F05; Secondary 35F10.
\endgraf
{\it Key words and phrases}. non-smooth flow, transport equation, Sobolev vector fields, divergence
}}
\author{\it Albert Clop, Renjin Jiang, Joan Mateu \& Joan Orobitg}
\date{}

\maketitle

\begin{center}
\begin{minipage}{12cm}\small
{\noindent{\bf Abstract} } In this paper,  we study flows associated to Sobolev vector fields with subexponentially integrable divergence. Our
approach is based on the transport equation following DiPerna-Lions \cite{dl89}.
A key ingredient is to use a quantitative estimate of solutions to the Cauchy problem of transport equation
to obtain the regularity of density functions.
\end{minipage}

\end{center}


\section{Introduction}

\noindent

Since the fundamental work by DiPerna-Lions \cite{dl89}, the study of flows associated to non-smooth vector fields has attracted intensive interest,
and has been found many applications in PDEs. The problem can be formulated as follows. Given a Sobolev (or more generally BV) vector field
$b:[0,T]\times \rn\to\rn$, does there exist a unique Borel map $X:[0,T]\times \rn\to\rn$, such that
\begin{equation}\label{ODE}
\dfrac{\partial }{\partial t}X(t,x)=b(t,X(t,x))
\end{equation}
for a.e. $x\in\rn$? If this ODE is well-posed, then how about the regularity of the solution $X$?\\
\\
In the seminal work by DiPerna and Lions \cite{dl89}, the existence of flows for Sobolev velocity fields with bounded divergence was established.
Their main ingredient was a careful analysis of the well posedness of the initial value problem for the linear transport equation,
\begin{equation}\label{transportPDE}
\begin{cases}
\dfrac{\partial u}{\partial t}+b\cdot \nabla u=0 & (0,T)\times\,\rn,\\
      u(0,\cdot)=u_0 &  \rn.
      \end{cases}
\end{equation}
In their arguments, the notion of renormalized solution was shown to be essential. Later, Ambrosio \cite{a04} extended the renormalization property
to the setting of bounded variation ($BV$) vector fields, and obtained the non-smooth flows by using some new tools from Probability and Calculus
of Variations. Crippa and De Lellis \cite{cdl08} used a direct approach to recover DiPerna-Lions' theory; see also Bouchut and Crippa \cite{bc13}.
Recently, in \cite{acf14}, Ambrosio, Colombo and Figalli developed a purely local theory on flows for non-smooth vector fields as a natural analogy
of the Cauchy-Lipschitz approach. \\
\\
Continuing our previous work about the transport equation \cite{cjmo}, in this paper we are concerned with the existence of flows for Sobolev
vector fields having sub-exponentially integrable divergence. Let us review some developments in this spirit. In \cite{d96}, Desjardins showed
existence and uniqueness of non-smooth flows for velocity fields having exponentially integrable divergence. Later, Cipriano and Cruzeiro
\cite{cc05} analyzed the flows for Sobolev vector fields with exponentially integrable divergence in the setting of Euclidean spaces equipped with
Gaussian measures; see \cite{af09} for related progresses in Wiener spaces.\\
\\
As already noticed in \cite{cc05,af09}, when the divergence of the velocity field is not bounded, the solution $X(t,\cdot)$ of equation \eqref{ODE}
still induces a quasi-invariant measure. This motivates the following definition.

\begin{defn}\label{defi}
Let $b: [0,T]\times \rn\to\rn$ be a Borel vector field, and $X,\,\tilde X: [0,T ]\times [0,T]\times  \rn\to\rn$ be Borel maps.
\begin{itemize}
\item[(i)] We say that $X$ is a forward flow associated to $b$ if for each $s\in [0,T]$ and almost every $x\in\rn$  the map
$t\mapsto\,|b(t, X(s, t, x))|$ belongs to $L^1(s,T )$ and
$$X(s,t, x) = x +\int_s^t b(r,X(s,r,x))\,dr.$$
We say that $\tilde X$ is a backward flow associated to $b$ if for each $t\in [0,T]$ and almost every $x\in\rn$  the map
{$s\mapsto\,|b(s, \tilde X(s, t, x))|$} belongs to {$L^1(0,t)$} and
$$\tilde X(s,t, x) = x -\int_s^t b(r,\tilde X(r,t,x))\,dr.$$
\item[(ii)] We say that $X$ is a regular flow associated to $b$ if:
\begin{enumerate}
\item $X$ is either a forward or a backward flow associated to $b$;
\item for  $0\le s\le t\le T$ the image measure $X(s,t,\cdot)_{\#}\,dx$ is absolutely continuous with respect to the Lebesgue measure $\,dx$.
\end{enumerate}
\item[(iii)]  We say that a forward flow $X$ associated with $b$ has the semigroup structure if for all $0\le r\le s\le t\le T$,
 it holds that
 $$X(s,t,X(r,s,x))=X(r,t,x), \ a.e.\ x\in\rn.$$
 We say that a backward flow $\tilde X$ associated with $b$ has the semigroup structure if for all $0\le r\le s\le t\le T$,
 it holds that
 $$\tilde X(r,s,\tilde X(s,t,x))=\tilde X(r,t,x), \ a.e.\ x\in\rn.$$
\end{itemize}
\end{defn}

\noindent
In this paper, we study regular flows as defined above. As in \cite{cc05}, in our arguments  sometimes it will be convenient to replace the Lebesgue
measure $dx$ by the Gaussian measure $\mu$ on $\rn$, i.e.,
$$\,d\mu(x)=\frac{1}{(2\pi)^{n/2}}\exp\left\{-\frac{|x|^2}{2}\right\}\,dx.$$
The distributional divergence of  a vector field $b$ with respect to the measure $\mu$ is then defined via
$$\mathrm{div}_\mu b(x)=\mathrm{div}\,b(x)-x\cdot b(x), \quad \forall\,x\in\rn,$$
that is, $\mathrm{div}_\mu$ is the adjoint of the gradient operator with respect to the measure $\mu$. This operator appears to be useful, among
other reasons because it commutes with the Ornstein-Uhlenbeck smoothing semigroup \cite{cc05, af09}. \\
\\
Our main result deals with existence and uniqueness of a regular flow for non-smooth vector fields with subexponentially integrable
divergence. Due
to the scheme of the proof, we found it convenient to state it in two steps. First, { for all $s\ge 0,$} we state the existence and uniqueness of
a flow for which all
$t$-advance maps {$X(s,t,\cdot)$} leave the Gaussian measure quasi-invariant, together with a quantitative estimate of this fact.
Secondly, we state that
the Lebesgue measure is also quasi-invariant, so that the flow we have found is indeed a regular flow. Moreover, we also state the semigroup
structure of the flow. The precise statement is as follows.

\newtheorem{mthm}{Main Theorem}
\renewcommand\themthm{}

\begin{mthm}
Let $b\in L^1(0,T;W^{1,1}_{\mathrm{loc}})$ satisfying
\begin{equation}\label{hyp-b-1}
\frac{|b(t,x)|}{1+|x|\log^+(|x|)}\in L^1(0,T;L^\infty),
\end{equation}
and
\begin{equation}\label{hyp-b-2}
\div_\mu b\in   L^1\left(0,T; \Exp_\mu\left(\frac{L}{\log L}\right)\right).
\end{equation}
Then the following statements hold.
\begin{enumerate}
\item[(a)] There exist a forward flow $X(s,t,x)$ and a backward flow  $\tilde X(s,t,x)$,  associated to $b$, which are unique
in the sense that,  for  $0\le s\le t\le T$:

(i)  $X(s,t,\tilde X(s,t,x))=\tilde X(s,t,X(s,t,x))=x,\ a.e. \,x\in \rn;$

(ii) the image measures $X(s, t,\cdot)_{\#}d\mu $ and $\tilde
X(s,t,\cdot)_{\#}d\mu$ are absolutely continuous with respect to
$d\mu$, and
$$\frac{d }{d\mu}\,(X(s,t,\cdot)_{\#}d\mu)=\exp\left\{\int_s^t -\mathrm{div}_\mu\,b(r,\tilde X(r,t,x))\,dr\right\}
\in L^{\Phi_\alpha}(\mu)\hspace{.3cm}\text{ for every
}0<\alpha<\alpha_0(s,t), $$
$$\frac{d }{d\mu}\,(\tilde X(s,t,\cdot)_{\#}d\mu)=\exp\left\{\int_s^t \mathrm{div}_\mu\,b(r,X(s,r,x))\,dr\right\}
\in L^{\Phi_\alpha}(\mu)\hspace{.3cm}\text{ for every
}0<\alpha<\alpha_0(s,t), $$
 where $\Phi_\alpha(r)=r\exp\{[\log^+(r)]^\alpha\}$ and
$\alpha_0(s,t)=\exp\left\{-16e^2\int_s^t\|\mathrm{div}_\mu
b(r,\cdot)\|_{\Exp_\mu(\frac{L}{\log L})}\,dr\right\}.$

\item[(b)] The unique flows $X(s,t,x)$ and $\tilde X(s,t,x)$ given in (a) are regular and have semigroup structure.
\end{enumerate}
\end{mthm}

\noindent
It is worth mentioning here that, under condition  \eqref{hyp-b-1}, the assumption \eqref{hyp-b-2} is equivalent to
$$\div b\in   L^1\left(0,T; \Exp_\mu\left(\frac{L}{\log L}\right)\right).$$
Concerning the optimality of \eqref{hyp-b-2}, it was proven in \cite[Section 6]{cjmo} that for every $\gamma>1$ there exists a velocity field $b$
with
\begin{equation}\label{optimal}
\div b\in L^1\left(0,T; \Exp_\mu\left(\frac{L}{\log^\gamma L}\right)\right)
\end{equation}
for which \eqref{ODE} admits infinitely many solutions $X$ satisfying (i) and (iii) in Definition \ref{defi}. However, we do not know if
\eqref{optimal} is sufficient or not to guarantee existence and uniqueness of solutions $X$ satisfying (i), (ii) and  (iii) in Definition \ref{defi}.\\
\\
Towards the proof of the Main Theorem, the main ingredient is the
following \textit{a priori} quantitative estimate for the density
function {$\frac{d }{d\mu}\,(X(s,t,\cdot)_{\#}d\mu)$}.
\begin{thm}\label{subexp-cr}
Let $b(t,\cdot)\in C^2(\rn)$ for each $t\in [0,T]$ and  satisfy
\eqref{hyp-b-1} and \eqref{hyp-b-2}. Then for $0\le s\le t\le T$, there exists a unique flow
$X(s,t,x)$ such that
$$\dfrac{\partial X(s,t,x)}{\partial t}=b(t,X(s,t,x)), \qquad  {X(s,s,x)=x.}$$
Moreover, for $0<\alpha<\exp\left\{-16e\int_s^t\beta(r)\,dr\right\}$,
$\beta(r)=\|\mathrm{div}_\mu b(r,\cdot)\|_{\Exp_\mu(\frac{L}{\log
L})}$,  the density function
{$K_{s,t}(x)=\frac{d}{d\mu}(X(s,t,x)_\#d\mu)$} belongs to
$L^{\Phi_\alpha}(\mu)$, and satisfies
\begin{equation}\label{est-jac}
\int_{\rn} {\Phi_\alpha}(K_{s,t}(x))\,d\mu(x)\le C(\alpha,s,t,
\div_\mu\,b).
\end{equation}
\end{thm}
Such estimate is established by means of a quantitative bound for
solutions to a Cauchy problem for the transport equation; see
Theorem \ref{quant} below. The use of this quantitative bound gives
a natural estimate of the density function. Moreover, as a
byproduct, our proof improves the integrability of the image measure
{$X(s,t,\cdot)_{\#}d\mu$}
 when $\div_\mu b$ is assumed to be exponentially integrable; see Theorem \ref{exp-cr} below and \cite{cc05,af09}.\\
\\
As it was for DiPerna and Lions scheme, well-posedness of the Cauchy problem \eqref{transportPDE} is an essential tool in our arguments. For
Sobolev vector fields $b$ satisfying the classical growth condition  $\frac{|b(t,x)|}{1+|x|}\in L^1(0,T; L^1)+L^1(0,T; L^\infty)$ and
$$\div b\in  L^1(0,T; L^\infty) + L^1\bigg(0,T; \Exp\bigg(\frac{L}{\log L}\bigg)\bigg),$$
the well-posedness of \eqref{transportPDE} in $L^\infty$ was established in \cite[Theorem 1]{cjmo}. Unfortunately, our Main Theorem does not
cover the assumption $\frac{|b(t,x)|}{1+|x|}\in L^1(0,T; L^1)$, and indeed we do not know if a flow does exist in this case. However, the assumption
on $\div\,b$ in the Main Theorem (also in Theorem \ref{tran-main} below) is less restrictive than it was in \cite[Theorem 1]{cjmo}. In other words,
our Theorem \ref{tran-main} about the well-posedness of \eqref{transportPDE} in $L^\infty$ slightly improves \cite[Theorem 1]{cjmo}. A similar
situation is given in Theorem \ref{quant-exp}, see Section 2 for details. \\
\\
From the result by Ambrosio-Figalli \cite{af09}, it looks like our requirements on the growth condition on $b$
are somehow natural. { Since }the image measure $X(s,t,\cdot)_{\#}d\mu$ is only slightly beyond $L^1$ integrable,
and to guarantee $b(t,X(s,t,x))\in L^1(s,T;L^1_\loc)$, we need to require that $b$ has at least exponential integrability.\\
\\
The paper is organized as follows.  In Section 2 we present the quantitative estimate of solutions to the transport equation (Theorem \ref{quant}),
and in Section 3, we use such estimate to deduce a priori estimate of the density function (Theorem \ref{subexp-cr}). In section 4, we give the proof
of part (a) of the Main Theorem. In the final section, we prove part (b) of the Main Theorem and give a stability result concerning the flows.
Throughout the paper, we denote by $C$ positive constants which are independent of the main parameters, but which may vary from line to line.

\section{Well-posedness of the transport equation in the Gaussian setting}

\hskip\parindent
We will need to use some Orlicz spaces and their duals. For the reader's convenience,
we recall here some definitions. See the monograph \cite{rr91} for the general theory of Orlicz spaces. Let
$$P:[0,\infty)\mapsto [0,\infty),$$
be an increasing homeomorphism onto $[0,\infty)$, so that $P(0)=0$ and $\lim_{t\to\infty}P(t)=\infty$. The Orlicz space $L^P$ is the set of
measurable functions $f$ for which the Luxembourg norm
$$\|f\|_{L^P}=\inf\left\{\lambda>0:\int_{\rr^n}P\left(\frac{|f(x)|}{\lambda}\right)\,dx\leq 1\right\}$$
is finite. In this paper we will be mainly interested in two particular families of Orlicz spaces.
Given $r,s\geq 0$, the first family corresponds to
$$P(t)=t\left(\log^+t\right)^{r}\,\left({\log^+ \log^+}t\right)^{s},$$
where $\log^+t:=\max\{1,\log t\}.$ The obtained $L^{P}$ spaces are known as \emph{Zygmund spaces},
and will be denoted from now on by $L\log^{r}L\,\log^{s}\log L $.
The second family is at the upper borderline. For $\gamma\geq 0$ we set
\begin{equation}\label{psi-function}
P(t)=\exp\left\{\frac{t}{(\log^+t)^\gamma}\right\}-1, \hspace{1cm} \,t\geq 0.
\end{equation}
Then we will denote the obtained $L^P$ by $\Exp(\frac{L}{\log^\gamma L} )$. If $\gamma=0$ or $\gamma=1$,
we then simply write $\Exp L$ and $\Exp(\frac{L}{\log L})$, respectively.
For each $\alpha>0$, throughout the paper, we denote by $\Phi_\alpha$ the Orlicz function
\begin{equation}\label{density-function}
\Phi_\alpha(t)=t\exp\left\{(\log^+t)^\alpha\right\}, \hspace{1cm} \,t\geq 0.
\end{equation}
\noindent
When changing the reference  measure from Lebesgue measure to the Gaussian measure, we will simply add $\mu$ to the notions of the spaces, as
$L\log L\log\log L(\mu)$, $\Exp_\mu\left(\frac{L}{\log L}\right)$, etc.\\
\\
The following lemma can be proved in the same way as \cite[Lemma
11]{cjmo}; see also \cite{rr91}.

\begin{lem}\label{duality}
(i) If $f\in L\log L\log\log L(\mu) $  and $g\in
\Exp_\mu(\frac{L}{\log L})$  then $fg\in L^1(\mu)$ and
$$\int_{\rn}|f(x)g(x)|\,d\mu\le 2\|f\|_{L\log L\log\log L(\mu)}\,\|g\|_{\Exp_\mu(\frac{L}{\log L})}.$$

(ii) If $f\in L\log L(\mu) $  and $g\in \Exp_\mu(L)$ then $fg\in
L^1(\mu)$ and
$$\int_{\rn}|f(x)g(x)|\,d\mu\le 2\|f\|_{L\log L(\mu)}\,\|g\|_{\Exp_\mu(L)}.$$
\end{lem}

\noindent
In this section we present a well-posedness result for the initial value problem for the transport equation in $L^\infty$. This is a new result, which
neither contains \cite[Theorem 1]{cjmo}, nor is contained in it. In order to state it, we write the transport equation in the Lebesgue case as
\begin{equation}\label{tran-eq-le}
\begin{cases}
\dfrac{\partial u}{\partial t}+b\cdot \nabla u=0 & (0,T)\times\,dx,\\
      u(0,\cdot)=u_0 &  \rn.
   \end{cases}\end{equation}
     and in the Gaussian case as
   \begin{equation}\label{tran-eq}
\begin{cases}
\dfrac{\partial u}{\partial t}+b\cdot \nabla u=0 & (0,T)\times\,d\mu,\\
      u(0,\cdot)=u_0 &  \rn.
   \end{cases}\end{equation}
   A function $u\in L^1(0,T;L^1_{\loc})$ is called a \textit{weak solution} to \eqref{tran-eq-le} if for each $\varphi\in C^\infty([0,T)\times\rr^n)$
   with compact support in  $[0,T)\times \rr^n$ it holds that
\begin{equation*}
-\int_0^T\int_{\rr^n} u\,\dfrac{\partial \varphi}{\partial t}\,dx\,dt-\int_{\rr^n} u_0 \,\varphi(0,\cdot)\,dx-\int_0^T\int_{\rr^n} u\,\div
(b\,\varphi)\,dx\,dt=0.
\end{equation*}
We also say that the problem \eqref{tran-eq-le} is \emph{well-posed} in $L^\infty(0,T;L^\infty)$ if weak solutions exist and are unique, for any
$u_0\in L^\infty$.\\
\\
Weak solutions of the transport equation \eqref{tran-eq} can be defined in a similar way. A simple observation is that a function $u\in
L^\infty(0,T;L^\infty)$ is a weak solution of \eqref{tran-eq-le} if and only if it is a weak solution of \eqref{tran-eq}. Indeed, if $u\in
L^\infty(0,T;L^\infty)$ is a weak solution of \eqref{tran-eq-le}, and $\varphi\in C^\infty_c([0,T)\times\rr^n)$ is a test function, then $ \frac{\varphi
(x)}{(2\pi)^{n/2}}\exp(-|x|^2/2)\in C^\infty_c([0,T)\times\rr^n)$, and so we can conclude that
\begin{equation*}
-\int_0^T\int_{\rr^n} u\,\dfrac{\partial \varphi}{\partial t}\,d\mu\,dt-\int_{\rr^n} u_0 \,\varphi(0,\cdot)\,d\mu-\int_0^T\int_{\rr^n}
u\left(\varphi\,\div_\mu b +b\cdot\nabla \varphi\right)\,d\mu(x)\,dt=0.
\end{equation*}
For the converse, we only need to use $ \frac{\varphi (x)}{(2\pi)^{n/2}}\exp(|x|^2/2)\in C^\infty_c([0,T)\times\rr^n)$ as a test function.\\

We now present our well posedness result for the transport equation
in the Gaussian setting.
\begin{thm}\label{tran-main}
Let $T>0$. Assume that  $b\in L^1(0,T;W^{1,1}_{\loc})$ satisfying \eqref{hyp-b-1} and \eqref{hyp-b-2}.
Then for each $u_0\in L^\infty$ there exists a unique weak solution
$u\in L^\infty(0,T;L^\infty)$ of the Cauchy problem for the transport equation \eqref{tran-eq}.
\end{thm}
\begin{proof}
Existence of solution follows immediately from \cite[Proposition
2.1]{dl89}, while uniqueness will follow from the following
stability estimate, Theorem \ref{quant}.
\end{proof}

The proof of the following two theorems is similar to \cite[Theorem 5]{cjmo}, so the proof will be omitted.

\begin{thm}\label{quant}
Let $T,M>0$ and $1\le p<\infty$. Suppose that $b\in L^1(0,T;W^{1,1}_{\loc})$ satisfies \eqref{hyp-b-1} and
\eqref{hyp-b-2}. Let $\epsilon\in (0,\frac12\exp(-e^{e+M}))$ satisfying
$$\exp\left\{-\exp\left\{\exp\left\{\log\log\log\frac 1{\epsilon}-32e\int_0^T\beta(s)\,ds\right\}\right\}\right\}<\frac12\exp(-e^{e+M}),$$
where $\beta(t)=\|\mathrm{div}_\mu b(t,\cdot)\|_{\Exp_\mu(\frac{L}{\log L})}$.
Then for each $u_0\in L^\infty(\mu)$ with
$\|u_0\|_{L^\infty(\mu)}\le M$ and $\|u_0\|^p_{L^p(\mu)}< \ez$,
the transport problem \eqref{tran-eq} has a unique solution $u\in L^\infty(0,T;L^\infty)$, moreover
it holds that
$$
\left|\log\log\log\left( \frac{1}{\|u(T,\cdot)\|^p_{L^p(\mu)}}\right)-\log\log\log\left(\frac{1}{\|u_0\|^p_{L^p(\mu)}}\right)\right|\leq
16e\int_0^T\beta(s)\,ds.
$$
\end{thm}
\begin{thm}\label{quant-exp}
Let $T,M>0$ and $1\le p<\infty$. Suppose that $b\in L^1(0,T;W^{1,1}_{\loc})$ satisfies
$$\frac{|b(t,x)|}{1+|x|\log^+(|x|)}\in L^1(0,T;L^\infty)+L^1(0,T;L^1)$$
 and
\begin{equation}\label{hyp-b-exp}
\mathrm{div}_\mu b\in L^1\left(0,T;\Exp_\mu\left(L\right)\right).
\end{equation}
Let $\epsilon\in (0,1/e)$ such that
$$\exp\left\{-\exp\left\{\log\log\frac 1{\epsilon}-8\int_0^T\beta(s)\,ds\right\}\right\}<\frac1{2({e+M})},$$
where $\beta(t)=\|\mathrm{div} b(t,\cdot)\|_{\Exp_\mu(L)}$.
Then for each $u_0\in L^\infty(\mu)$ with
$\|u_0\|_{L^\infty(\mu)}\le M$ and $\|u_0\|^p_{L^p(\mu)}< \ez$,
the transport problem \eqref{tran-eq} has a unique solution $u\in L^\infty(0,T;L^\infty)$, moreover
it holds that
$$
\left|\log\log\left( \frac{1}{\|u(T,\cdot)\|^p_{L^p(\mu)}}\right)-
\log\log\left(\frac{1}{\|u_0\|^p_{L^p(\mu)}}\right)\right|\leq 4\int_0^T\beta(s)\,ds.
$$
\end{thm}

\section{A priori estimates of the Jacobian}
\hskip\parindent In this section, we give a priori estimates of the density functions when
we assume that the vector field is smooth. Recall that ${\Phi_\alpha}(s)=s\exp\{[\log^+(s)]^\alpha\}$
 is given in \eqref{density-function}.

\begin{proof}[Proof of Theorem \ref{subexp-cr}]
The existence and uniqueness of the flow is an immediate consequence
of the assumption that $b(t,\cdot)\in C^2(\rn)$ for each $t\in
[0,T]$ and satisfies
$$\frac{|b(t,x)|}{1+|x|\log^+|x|}\in L^1(0,T;L^\infty).$$
Moreover, the { forward flow associated to $b,$ } $X(s,t,x)$,  is locally Lipschitz for each $0\le s\le t\le T$. See Hale \cite{ha80} for instance.

Let us estimate the density function. Obviously, it holds that
$$\int_\rn K_{s,t}(x)\,d\mu(x)=\int_\rn\,d\mu(x)=1,$$
i.e., $\|K_{s,t}\|_{L^1(\mu)}=1$ for each $t\in [s,T]$. As a consequence,
$$\mu(\{x:\,K_{s,t}(x)>\lambda\})\le \frac{1}{\lambda}$$
for all $\lambda>0$ and $t\in [s,T]$.

 Let $k_0\in\cn$ be large enough such that
$$\exp\left\{-\exp\left\{\exp\left\{\log\log\log 2^{k_0}-32e\int_0^{T}\beta(r)\,dr\right\}\right\}\right\}
<\frac12\exp(-e^{2e}),$$
where $\beta(r)=\|\mathrm{div}_\mu b(r,\cdot)\|_{\Exp_\mu(\frac{L}{\log L})}$. Obviously, $k_0$ only depends on
$\|\mathrm{div}_\mu b(r,\cdot)\|_{\Exp_\mu(\frac{L}{\log L})}$.

Fix  $0\le s_0\le t_0\le T$. For each $k>k_0$, let
$$E_k=\{x\in\rn:\,2^{k-1}<K_{s_0,t_0}(x)\le 2^k\},$$
and $u_k(x)=\chi_{E_k}(x)$, where $\chi_E$ denotes the characteristic function of the set $E$. Then $u_k\in L^1(\mu)\cap L^\infty(\mu)$ with $\|u_k\|_{L^\infty(\mu)}\le 1$
and $\|u_k\|_{L^1(\mu)}\le 2^{1-k}$.

{\bf Claim:}  $u(s,x):=u_k(X(s,t_0,x))$ is the unique solution in the
Gaussian setting to the backward equation
\begin{equation*}
\begin{cases}
\dfrac{\partial u}{\partial {s}}+b\cdot \nabla u=0 & (0,t_0)\times\,\,d\mu,\\
      u(t_0,\cdot)=u_k &  \rn.
   \end{cases}\end{equation*}
   {\it Proof of the Claim:}  Let   $\varphi\in C^\infty_c((0,t_0]\times\rr^n)$  be a test function.
   Since $b(t,\cdot)\in C^2(\rn)$ for each $t\in [0,T]$, we know that  the density function $K_{s,t}$ satisfies
   \begin{equation}\label{density-funct-smooth}
   K_{s,t}(x)=\exp\left\{\int_s^t -\mathrm{div}_\mu\,b(r,\tilde X(r,t,x))\,dr\right\},
\end{equation}
where
$$\tilde X(s,t,x)=x-\int_s^t b(r,\tilde  X(r,t,x))\,dr$$
is the inverse map of $X(s,t,x)$; see \cite[Theorem 2.1]{cc05} or
\cite{af09}.
By using  change of
variables and integration by parts, we obtain that
\begin{eqnarray*}
&&\int_0^{t_0}\int_{\rr^n} u\,\dfrac{\partial \varphi}{\partial s}\,d\mu\,ds=\int_0^{t_0}\int_{\rr^n} u_k(X(s,t_0,x))\,\dfrac{\partial \varphi(s,x)}{\partial
s}\,d\mu\,ds\\
&&\hs=\int_0^{t_0}\int_{\rr^n} u_k(y)\,\dfrac{\partial \varphi(s,z)}{\partial
s}\bigg|_{z=\tilde X(s,t_0,y)}K_{s,t_0}(y)\,d\mu\,ds\\
&&\hs=\int_0^{t_0}\int_{\rr^n} u_k(y)\left[\,\dfrac{\partial \varphi(s,\tilde X(s,t_0,y))}{\partial s}-\nabla \varphi(s,\tilde X(s,t_0,y))\cdot
b(s,\tilde X(s,t_0,y))\right]K_{s,t_0}(y)\,d\mu\,ds\\
&&\hs=\int_{\rr^n} u_k \,\varphi(t_0,\cdot)\,d\mu-\int_0^{t_0}\int_{\rr^n} u_k(y) \varphi(s,\tilde X(s,t_0,y))K_{s,t_0}(y)\mathrm{div}_\mu\,b(s,\tilde X(s,t_0,y))\,\,d\mu\,dt\\
&&\hs\hs   -\int_0^{t_0}\int_{\rr^n} u_k(y)\nabla \varphi(s,\tilde X(s,t_0,y))\cdot b(s,\tilde X(s,t_0,y))K_{s,t_0}(y)\,d\mu\,ds\\
&&\hs=\int_{\rr^n} u_k
\,\varphi(t_0,\cdot)\,d\mu-\int_0^{t_0}\int_{\rr^n} u_k(X(s,t_0,x))
\left[\varphi(s,x)\mathrm{div}_\mu\,b(s,x)+\nabla \varphi(s,x)\cdot
b(s,x)\right]\,\,d\mu\,ds,
\end{eqnarray*}
which verifies the Claim.  Above, { in the third equality, we have used that  $\partial \tilde X(s,t,x)/\partial s=
b(s,\tilde X(s,t,x))$ and
\begin{eqnarray*}
&&\,\dfrac{\partial \varphi(s,\tilde X(s,t_0,y))}{\partial s}=\,\dfrac{\partial \varphi(s,z)}{\partial
s}\bigg|_{z=\tilde X(s,t_0,y)}+\nabla \varphi(s,\tilde X(s,t_0,y))\cdot
b(s,\tilde X(s,t_0,y)).
\end{eqnarray*}}


By Theorem \ref{quant} and the choose of $k_0$, we find that for each $s\in [0,t_0]$ it holds
$$
\left|\log\log\log\left( \frac{1}{\|u(s,\cdot)\|_{L^1(\mu)}}\right)
-\log\log\log\left(\frac{1}{\|u_k\|_{L^1(\mu)}}\right)\right|\leq 16e\int_s^{t_0}\beta(r)\,dr,
$$
which implies that
$$
\exp\left\{-16e\int_{s}^{t_0}\beta(r)\,dr\right\}\le \dfrac{\log\log\left(
\frac{1}{\|u(s,\cdot)\|_{L^1(\mu)}}\right)}{\log\log\left(\frac{1}{\|u_k\|_{L^1(\mu)}}\right)}
\leq \exp\left\{16e\int_s^{t_0}\beta(r)\,dr\right\}.
$$
Hence, we can conclude that
$$
\left(\log\frac{1}{\|u_k\|_{L^1(\mu)}}\right)^{\exp\left\{-16e\int_{s}^{t_0}\beta(r)\,dr\right\}}\le
\log\frac{1}{\|u(s,\cdot)\|_{L^1(\mu)}}
\leq \left(\log\frac{1}{\|u_k\|_{L^1(\mu)}}\right)^{\exp\left\{16e\int_{s}^{t_0}\beta(r)\,dr\right\}}.
$$
The choose of $u$ implies
that
$$\|u(s_0,\cdot)\|_{L^1(\mu)}=\int_{E_k}K_{s_0,t_0}(x)\,d\mu(x),$$
and hence,
$$
\left(\log\frac{1}{\mu(E_k)}\right)^{\exp\left\{-16e\int_{s_0}^{t_0}\beta(r)\,dr\right\}}\le
\log\frac{1}{2^{k-1}\mu(E_k)}=\log\frac{1}{2^{k-1}}+\log\frac{1}{\mu(E_k)}.
$$
A direct calculation gives
$$
\log\frac{1}{\mu(E_k)}\ge \log {2^{k-1}}+
\left[\log {2^{k-1}}\right]^{\exp\left\{-16e\int_{s_0}^{t_0}\beta(r)\,dr\right\}}
$$
Therefore, we can conclude that,
$$\mu(E_k)\le \exp\left\{-\log {2^{k-1}}-
\left[\log {2^{k-1}}\right]^{\exp\left\{-16e\int_{{s_0}}^{t_0}\beta(s)\,ds\right\}}\right\}
\le \frac 1{2^{k-1}} \exp\left\{-\left(\log 2^{k-1}\right)^{\exp\left\{-16e\int_{s_0}^{t_0}\beta(r)\,dr\right\}}\right\}.$$

For an arbitrary $\alpha\in
(0,\exp\left\{-16e\int_{s_0}^{t_0}\beta(r)\,dr\right\})$, we have that
\begin{eqnarray*}
&&\int_{\rn} K_{{s_0,t_0}}(x)\exp\{[\log^+ K_{{s_0,t_0}}(x) ]^\alpha\}\,d\mu(x)\\
&&
\le\int_{\rn} 2^{k_0}\exp\{[\log^+ 2^{k_0}]^\alpha\}\,d\mu(x)+\sum_{k>k_0}\int_{E_k} 2^k\exp\{[\log^+ (2^k) ]^\alpha\}\,d\mu(x)\\
&&\le 2^{k_0}\exp\{[\log^+ 2^{k_0}]^\alpha\}+\sum_{k>k_0}\mu({E_k})2^k
\exp\{[\log^+ (2^k) ]^\alpha\}\\
&&\le 2^{k_0}\exp\{[\log^+ 2^{k_0}]^\alpha\}+\sum_{k>k_0}2
\exp\left\{[\log^+ (2^k) ]^\alpha-\left(\log 2^{k-1}\right)^{\exp\left\{-16e\int_{s_0}^{t_0}\beta(r)\,dr\right\}}\right\}\\
&&\le C(\alpha,s_0,t_0,\mathrm{div}_\mu b).
\end{eqnarray*}
This completes the proof.
\end{proof}

In the same way, using Theorem \ref{quant-exp}, we can prove the following quantitative estimate for vector fields with distributional
divergence in
$\Exp_\mu(L)$.

\begin{thm}\label{exp-cr}
Let $b(t,\cdot)\in C^2(\rn)$ for each $t\in [0,T]$ such that
$$\frac{|b(t,x)|}{1+|x|\log^+|x|}\in L^1(0,T;L^\infty),$$
and $\div_\mu b\in L^1(0,T;\Exp_\mu(L))$.
 Then for $0\le s\le t\le T$, there exists a unique flow
$X(s,t,x)$ such that
$$\dfrac{\partial X(s,t,x)}{\partial t}=b(t,X(s,t,x)).$$
Moreover, for $0\le s\le t\le T$ and
 each $p\in
[1,\frac{1}{1-\exp(-4\int_s^t\beta(r)\,dr)}),$
$\beta(r)=\|\mathrm{div}_\mu b(r,\cdot)\|_{\Exp_\mu(L)}$, the
density function $K_{s,t}(x)=\frac{d}{d\mu}(X(s,t,{x})_\#d\mu)$ belongs to
$L^p(\mu)$ and satisfies
\begin{equation*}
\int_{\rn} [K_{s,t}(x)]^p\,d\mu(x)\le C(p,s,t,
\div\,b).
\end{equation*}
\end{thm}
\begin{rem}\rm
 Our method to prove the integrability of the density functions yields a sharper estimate than  those from \cite{af09,cr83,cc05}.
It is worth to note that our proof yields that integrability of the density functions has some semigroup property, which is natural.
\end{rem}

\section{Flow in the Gaussian setting}

\noindent
In this section, we will prove part (a) of the Main Theorem. To do this, let us recall the Ornstein-Uhlenbeck semigroup $P_s$. For each $s>0$ and
$f\in L^1(\mu)$, $P_sf(x)$ is defined by
$$P_sf(x)=\int_\rn f(e^{-s}x+\sqrt{1-e^{-2s}}y)\,d\mu(y).$$
 Among other properties of the semigroup $P_s$,
we will need the following:

\begin{enumerate}
\item[(i)] $\div_\mu\,(P_sb)=e^{s}P_s(\div_\mu\,b)$.

\item[(ii)] For each $p\in [1,\infty]$, it holds
$$\|P_sf\|_{L^p(\mu)}\le \|f\|_{L^p(\mu)}.$$

\item[(iii)] For each convex function $\Phi$ on $[0,\infty)$, $\Phi(0)=0$, $\lim_{s\to\infty}\frac{\Phi(s)}{s}=\infty$, it holds
that
$$\|P_sf\|_{L^\Phi(\mu)}\le \|f\|_{L^\Phi(\mu)}.$$
\end{enumerate}
The first two properties can be found from  Bogachev \cite{bo98}, and the third one is a consequence of
(ii) and  Jensen's inequality. Indeed, Jensen's inequality and the $L^1$-boundedness of $P_s$ imply
\begin{eqnarray*}
\int_\rn\Phi\left(\frac{P_sf}{\lambda}\right)\,d\mu&&\le
\int_\rn\int_\rn \,\Phi\left(\frac{f(e^{-s}x+\sqrt{1-e^{-2s}}y)}{\lambda}\right)\,d\mu(y)\,d\mu(x)
\le \int_\rn \Phi\left(\frac{f(x)}{\lambda}\right)\,d\mu(x).
\end{eqnarray*}

We will use the transport
equation theory by DiPerna-Lions \cite{dl89} and
follow some methods used  by  Cipriano-Cruzeiro \cite{cc05}. Due to the fact that
 the divergence of the vector is only sub-exponentially integrable, we need to
overcome some technical difficulties.

In what follows, we will always let $b\in L^1(0,T;W^{1,1}_{\mathrm{loc}})$ that satisfies
$$\frac{|b(t,x)|}{1+|x|\log^+|x|}\in L^1(0,T;L^\infty), $$
and $\div\,b\in L^1(0,T;\Exp_\mu(\frac{L}{\log L}))$.
It follows by an easy calculation that
$$\div_\mu\,b=\div\,b-x\cdot b\in L^1\left(0,T;\Exp_\mu\left(\frac{L}{\log L}\right)\right).$$
For each $\ez>0$, let $b_\ez=P_\ez b$.

\begin{lem}\label{molifier}
For each $\ez>0$, $P_\ez b\in C^\infty(\rn)$ satisfies
$$\frac{|P_\ez b(t,x)|}{1+|x|\log^+|x|}\in L^1(0,T;L^\infty).$$
\end{lem}
\begin{proof}
By making change of variables, we see that
\begin{eqnarray*}
P_\ez b(t,x)&&=\int_\rn b(t,e^{-\ez}x+\sqrt{1-e^{-2\ez}}y)\frac{1}{(2\pi)^{n/2}}\exp\left\{-\frac{|y|^2}{2}\right\}\,dy\\
&&=\frac{1}{(2\pi)^{n/2}(1-e^{-2\ez})^{n/2}}\int_\rn b(t,z)\exp\left\{-\frac{|z-e^{-\ez}x|^2}{2(1-e^{-2\ez})}\right\}\,dz.
\end{eqnarray*}
Then it is obvious that $P_\ez b(t,x)\in C^\infty(\rn)$ for each $t>0$. To see that
$$\frac{P_\ez b(t,x)}{1+|x|\log^+|x|}\in L^1(0,T;L^\infty),$$
it suffices to show that for each $t>0$
$$\left\|\frac{P_\ez b(t,x)}{1+|x|\log^+|x|}\right\|_{L^\infty}\le C \left\|\frac{b(t,x)}{1+|x|\log^+|x|}\right\|_{L^\infty}.$$
By the fact $\log(a+b)\le \log a+\log b$ for $a,b\ge 2$, we see that
\begin{eqnarray*}
|P_\ez b(t,x)|&&\le \int_\rn |b(t,e^{-\ez}x+\sqrt{1-e^{-2\ez}}y)|\frac{1}{(2\pi)^{n/2}}\exp\left\{-\frac{|y|^2}{2}\right\}\,dy\\
&&\le \left\|\frac{b(t,\cdot)}{1+|\cdot|\log^+|\cdot|}\right\|_{L^\infty}\int_\rn (1+|z|\log^+|z|)\bigg|_{ z=e^{-\ez}x+\sqrt{1-e^{-2\ez}}y}\,d\mu(y),
\end{eqnarray*}
where
\begin{eqnarray*}
\int_\rn (1+|z|\log^+|z|)\bigg|_{ z=e^{-\ez}x+\sqrt{1-e^{-2\ez}}y}\,d\mu(y)&&\le \int_\rn C(1+|x|\log^+|x|+|y|\log^{+}|y|)\,d\mu(y)\\
&&\le  C(1+|x|\log^+|x|),
\end{eqnarray*}
where $C$ does not depend on $\ez$. The proof is completed.
\end{proof}

Therefore, for each $\ez>0$, it follows from Lemma \ref{molifier} that
$b_\ez$ satisfies the requirements from Theorem \ref{subexp-cr} uniformly in $\ez$.
Denote by $X_\ez(s,t,x)$ the unique flow arising from the equation
$$\dfrac{\partial X_\ez(s,t,x)}{\partial t}=b_\ez(t,X_\ez(s,t,x)).$$
Denote by $K_{s,t,\ez}(x)$ the density function of $X_\ez(s,t,\cdot)_{\#}\,d\mu$.
The existence of the flow $X(s,t,x)$ will follow by establishing  an accumulation point of $\{X_\ez(s,t,x)\}_{\ez}$
via the following several steps.

Given a sequence $X_k$ of functions defined on some measurable space
$(\mathscr{M},\nu)$ with values in a Banach space $\mathscr{N}$ (endowed with the norm $\|\cdot\|$), we say that
$X_k$ converges to $X$ in $L^0(\nu)$ if for each fixed $\gz > 0$ it holds
$$\nu(\{x\in \mathscr{M}:\, \|X_k(x)-X(x)\|>\gz\})\to 0, \quad \ \mathrm{as}\  k\to\infty.$$

In what follows, let  $\mathcal{L}^1$ be the one dimensional Lebesgue measure.
\begin{lem}\label{mc-flow}
Let $0\le s\le t\le T$. There exist a subsequence $\{\epsilon_k\}_{k\in\cn}$ and a Borel map $X(s,t,x)$ such that:
\begin{itemize}
\item[(i)] $X_{\ez_k}(s,\cdot,\cdot) $ converges to $X(s,\cdot,\cdot)$ as $k\to\infty$, both in $L^0(\mathcal{L}^1\times\mu)$ and almost everywhere on  $[s,T]\times \rn$.\item[(ii)] For each   fixed $t\in [s,T]$, $X_{\ez_k}(s,t,\cdot)$ converges to $X(s,t,\cdot)$ as $k\to\infty$, both in $L^0(\mu)$ and almost everywhere on $\rn$.
\end{itemize}
\end{lem}
\begin{proof}
Let $\beta$ be a continuous and bounded function on $\rr$. Denote
$X^i_{\ez}(s,t,x)$ the $i$-th component of $X_{\ez}(s,t,x)$. Then
$\beta(X^i_{\ez}(s,t,x))$ and $\beta(X^i_{\ez}(s,t,x))^2$ are bounded
sequences in $L^\infty(s,T;L^\infty)$. By the weak-$\ast$
convergence of $L^\infty(s,T;L^\infty)$, we see that there exists a
subsequence $\ez_k$ such that $\beta(X^i_{\ez_k}(s,t,x))$
 and $\beta(X^i_{\ez_k}(s,t,x))^2$
 converge in weak-$\ast$ topology of $L^\infty(s,T;L^\infty)$ to $v_\beta^i$ and $w_\beta^i$, respectively.

 On the other hand, $\beta(X^i_{\ez}(s,t,x))$ and $\beta(X^i_{\ez}(s,t,x))^2$ are bounded solutions to the transport equation
 corresponding to the final values $\beta(x_i)$ and $\beta(x_i)^2$, respectively; see the proof of Theorem \ref{subexp-cr}.

By using the well-posedness of the transport equation, Theorem \ref{tran-main}, and the renormalization property
 of solutions in $L^\infty(s,T;L^\infty)$ (cf. \cite{dl89,a04}), we can conclude that
 $v_\beta^i$ and $w_\beta^i$ are bounded solutions to the transport equation
 with vector fields $b$ corresponding to the initial values $\beta(x_i)$ and $\beta(x_i)^2$, respectively,
 and therefore $(v_\beta^i)^2=w_\beta^i$.

 Then, by the fact $1\in L^1(\mu)$, we can conclude that for each $t\in [s,T]$ it holds
\begin{equation}\label{measure-beta}
\lim_{k\to\infty}\int_\rn [\beta(X^i_{\ez}(s,t,x))-v_\beta^i]^2\,d\mu=0.
\end{equation}

 Now we prove that the arbitrariness of $\beta$ implies that $X^i_{\ez_k}(s,t,x)$
 converges in measure to some function $X^i(s,t,x)$.
 Indeed, by Lemma \ref{molifier} we have that
 $$\frac{|b_\ez(t,x)|}{1+|x|\log^+|x|}\in L^1(0,T;L^\infty)$$
 and hence, $b_\ez(t,x)\in L^1(0,T;\Exp_\mu(L))$, while by Theorem
 \ref{subexp-cr} and $L^{\Phi_\alpha}(\mu)\subset L\log L(\mu)$ for
 any $\alpha>0$,
 we see that $K_\ez\in L^\infty(0,T;L\log L(\mu))$. These together with Lemma \ref{duality} imply that
 \begin{eqnarray*}
 \|X^i_{\ez_k}(s,t,\cdot)\|_{L^1(\mu)}&&\le \int_\rn \left|x_i+\int_s^tb_\ez(r,X_\ez(s,r,x))\,dr\right|\,d\mu(x)\\
 &&\le C+\int_\rn\int_s^T |b_\ez(r,x)|\,K_{s,r,\ez}(x)\,d\mu\,dr\\
 &&\le C+2 \int_s^T\|b_\ez\|_{\Exp_\mu(L)}\|K_{s,r,\ez}\|_{L\log L(\mu)}\,dr\\
 &&\le C,
 \end{eqnarray*}
i.e., $X_{\ez_k}(s,\cdot,\cdot)\in L^\infty(s,T;L^1(\mu))$, and $X_{\ez_k}(s,t,\cdot)\in L^1(\mu)$ for each $t$,
uniformly in $\ez$.

 Denote by $\nu$ the product measure $\mathcal{L}^1\times\mu$ on $[s,T]\times \rn$. Given a fixed $\gamma>0$,
 for each $\delta>0$, there exists an $M>0$ such that for all $\ez_k$,
 $$\nu(\left\{(t,x):\,|X_{\ez_k}(s,t,x)|>M\right\})<\delta.$$
 On the other hand, let $\beta_M\in C^1(\rr,\rr)$ such that $\beta_M:\,\rr\mapsto [-2M,2M]$ and $\beta_M(t)=t$
 for all $|t|\le M$. Then from \eqref{measure-beta} we see that there exists $k_0\in \cn$, such that for all $k,j>k_0$, it holds that$$
\aligned
\nu(\{(t,x): |X^i_{\ez_k}(s,t,x)-X^i_{\ez_j}(s,t,x)|>\gamma\})
&\leq \nu(\{(t,x): |X^i_{\ez_k}(s,t,x)|>M\})\\&+\nu(\{(t,x): |X^i_{\ez_j}(s,t,x)|>M\})\\
&+\nu(\{(t,x):\,|\beta_M(X^i_{\ez_k}(s,t,x))-\beta_M(X^i_{\ez_j}(s,t,x))|>\gamma\})<3\delta
\endaligned$$
and so we can conclude that $\{X^i_{\ez_k}\}_k$ is a Cauchy sequence in measure. Therefore, $X^i_{\ez_k}(s,t,x)$ converges in measure to some
function $X^i(s,t,x)$.

 Passing to a further subsequence if necessary, we can conclude that $X_{\ez_k}(s,t,x)$ converges
 in $L^0(\mathcal{L}^1\times\mu)$ and almost everywhere to $X(s,t,x)$ on $[s,T]\times \rn$.
 Moreover, it follows that for each $t\in [s,T]$, $X_{\ez_k}(s,t,x)$ converges in $L^0(\mu)$ and almost everywhere to $X(s,t,x)$.
\end{proof}

\begin{lem}\label{density-est}
Let $X(s,t,x)$ be as in Lemma \ref{mc-flow}. Under the assumptions of
the Main Theorem, for each $0\le s\le t\le T$, the image measure
$X(s,t,\cdot)_{\#}d\mu$ is absolutely continuous with respect to
$\mu$. Moreover,  the density function $K_{s,t}(x)=\frac{d}{d\mu}
(X(s,t,{x})_{\#}d\mu)$ belongs to the Orlicz space $L^{\Phi_\alpha}(\mu)$
for each $0<\alpha<\exp\left\{-16e^2\int_s^t\beta(r)\,dr\right\}$,
where $\beta(r)=\|\mathrm{div}_\mu
b(r,\cdot)\|_{\Exp_\mu(\frac{L}{\log L})}$.
\end{lem}
\begin{proof}
Since $b_\ez=P_\ez b$, by the property of the  Ornstein-Uhlenbeck semigroup, we see that for each $\ez<1$ and each $t\in [s,T]$,
it holds
$$\|\mathrm{div}_\mu b_\ez(t,\cdot)\|_{\Exp_\mu(\frac{L}{\log L})}\le e\|\mathrm{div}_\mu b(t,\cdot)\|_{\Exp_\mu(\frac{L}{\log L})}.$$
For each $t\in [s,T]$ and each $0<\alpha<\exp\left\{-16e^2\int_s^t\beta(r)\,dr\right\}$,
by Theorem \ref{subexp-cr},  we see that the density function
of  $K_{s,t,\ez}(x)=\frac{d}{d\mu}(X_\ez(s,t,{x})_{\#}d\mu)$
 is uniformly bounded in $L^{\Phi_\alpha}(\mu)$. Therefore, there exists a subsequence $\{\ez_k\}$ and $K_{s,t}\in L^{\Phi_\alpha}(\mu)$
 such that
{ $$ K_{s,t,\ez_k}\rightharpoonup K_{s,t} \  \mathrm{ in } \ L^{\Phi_\alpha}(\mu),$$}
 i.e., {$K_{s,t,\ez_k}$} weakly converges to {$K_{s,t}$} in $L^{\Phi_\alpha}(\mu)$.

 Finally, for each compactly supported continuous function $\psi$, we see that
 \begin{eqnarray*}
\int_\rn \psi(X(s,t,x))\,d\mu(x)&&=\lim_{k\to\infty}\int_\rn \psi(X_{\ez_k}(s,t,x))\,d\mu(x)\\
&&=\lim_{k\to\infty} \int_\rn
\psi(x)K_{s,t,\ez_k}(x)\,d\mu(x)=\int_\rn \psi(x)K_{s,t}(x)\,d\mu(x),
 \end{eqnarray*}
as desired.
\end{proof}

\begin{lem}\label{measure-control}
Let $X(s,t,x)$ be as in Lemma \ref{mc-flow}. Under the assumptions of
the Main Theorem, for each open set $E$ with sufficient small
$\mu$-measure, it holds that for $0\le s\le t\le T$
$$\log\log\log\dfrac{1}{\int_\rn \chi_E(X(s,t,x))\,d\mu}\ge \log\log\log\dfrac{1}{\mu(E)}- 16e^2\int_s^t\beta(r)\,dr.$$
\end{lem}
\begin{proof}
Since $E$ is an open set, by the a.e. convergence of $X_{\ez_k}(s,t,x)$, it is easy to see that
$$\liminf_{k\to\infty}\chi_E(X_{\ez_k}(s,t,x))\ge \chi_E(X(s,t,x)),\quad a.e.\ x\in\rn.$$
Therefore it follows from Fatou Lemma that
$$\int_\rn \chi_E(X(s,t,x))\,d\mu\le \int_\rn \liminf_{k\to\infty}\chi_E(X_{\ez_k}(s,t,x))\,d\mu\le
 \liminf_{k\to\infty}\int_\rn\chi_E(X_{\ez_k}(s,t,x))\,d\mu .$$
Since $$\|\mathrm{div}_\mu b_\ez\|_{\Exp_\mu(\frac{L}{\log L})}\le
e\|\mathrm{div}_\mu b\|_{\Exp_\mu(\frac{L}{\log L})},$$ by Theorem
\ref{quant} we know that for each $k$, it holds
$$\left|\log\log\log\dfrac{1}{\int_\rn \chi_E(X_{\ez_k}(s,t,x))\,d\mu}- \log\log\log\dfrac{1}{\mu(E)}\right|\le  16e^2\int_s^t\beta(r)\,dr,$$
which together with the last estimate completes the proof.
\end{proof}

\begin{lem}\label{measure-convergence}
Let $X(s,t,x)$ be as in Lemma \ref{mc-flow}. Under the assumptions of
the Main Theorem, for each measurable vector field
$F:\,[s,T]\times\rn\mapsto\rn$, it holds
$$F(t,X_{{\ez_k}}(s,t,x))\to F(t,X(s,t,x)) \  \mbox{in}\ L^0(\mathcal{L}^1\times\mu);$$
and for measurable function $F:\,\rn\mapsto\rn$, it holds for each $t\in [s,T]$ that
$$F(X_{{\ez_k}}(s,t,\cdot))\to F(X(s,t,\cdot)) \  \mbox{in}\ L^0(\mu).$$
\end{lem}
\begin{proof} We only prove the second statement, since the first one can be proved in the same way.
By the Egorov Theorem, for each $\delta>0$, there exists a measurable set $E_\delta$  such that
$\mu(\rn\setminus E_\delta)<\delta$ and $F$ is uniformly continuous on $E_\delta$.

On the other hand, by  using the Egorov Theorem again and the fact $X_{{\ez_k}}(s,t,x)$ converges in measure to $X(s,t,x)$, we find that  there exists
$\widetilde E_\delta$ such that $\mu(\rn\setminus \widetilde E_\delta)<\delta$ and
$X_{{\ez_k}}(s,t,x)$ converges uniformly to $X(s,t,x)$ on $\widetilde E_\delta$.

Therefore, for a fixed constant $c$,
\begin{eqnarray*}
&&\mu\left(\left\{x:\, |F(X_{\ez_k}(s,t,x))-F(X(s,t,x))|>c\right\}\right)\\
&&\quad\le
\mu(\rn\setminus\widetilde E_\delta)+\mu(\left\{x:\, X(s,t,x)\in \rn\setminus E_\delta\right\})+\mu(\left\{x:\, X_{\ez_k}(s,t,x)\in \rn\setminus
E_\delta\right\}) \\
&&\quad\quad+\mu\left(\left\{x\in \widetilde E_\delta, X_{\ez_k}(s,t,x),X(s,t,x)\in E_\delta:
\, |F(X_{\ez_k}(s,t,x))-F(X(s,t,x))|>c\right\}\right).
\end{eqnarray*}

Notice that by Theorem \ref{tran-main}, we have that
$$\mu(\left\{x:\, X_{\ez_k}(s,t,x)\in \rn\setminus E_\delta\right\})\le \exp\left\{-\left(\log\frac
1\delta\right)^{\exp\left\{-C\int_s^t\beta(r)\,dr\right\}}\right\}$$
uniformly in $k$, and by Lemma \ref{measure-control}
$$\mu(\left\{x:\, X(s,t,x)\in \rn\setminus E_\delta\right\})\le
\mu\left(\left\{x:\, X(s,t,x)\in \widetilde{\rn\setminus E_\delta}\right\}\right)\le
\exp\left\{-\left(\log\frac 2\delta\right)^{\exp\left\{-C\int_s^t\beta(r)\,dr\right\}}\right\},$$
where $\widetilde{\rn\setminus E_\delta}$ is an open set containing $\rn\setminus E_\delta$
satisfying
$$\mu(\widetilde{\rn\setminus E_\delta})\le 2\mu(\rn\setminus E_\delta).$$

By choosing large enough $k$, we have
$$\mu\left(\left\{x\in \widetilde E_\delta, X_{\ez_k}(s,t,x),X(s,t,x)\in E_\delta:
\, |F(X_{\ez_k}(s,t,x))-F(X(s,t,x))|>c\right\}\right)=0.$$

Therefore, for each $\gz>0$, by choosing sufficiently small $\delta$, we see that there exists $k_\gz$, such that
for each $k>k_\gz$, it holds
\begin{eqnarray*}
&&\mu\left(\left\{x:\, |F(X_{\ez_k}(s,t,x))-F(X(s,t,x))|>c\right\}\right)<\gz,
\end{eqnarray*}
which completes the proof.
\end{proof}

\begin{lem}\label{flow-eq}
Let $X(s,t,x)$ be as in Lemma \ref{mc-flow}. Under the assumptions of
the Main Theorem, for $0\le s\le t\le T$, we have
$$X(s,t,x)=x+\int_{s}^tb(r,X(s,r,x))\,dr$$
for a.e. $x\in\rn$.
\end{lem}
\begin{proof}
It suffices to prove that for each $s\in [0,T)$,
$$\int_\rn\int_{s}^T|b_{\ez_k}(r,X_{\ez_k}(s,r,x))-b(r,X(s,r,x))|\,dr\,d\mu\to 0\ \mbox{as}\  \ \ez_k\to 0.$$
Write
\begin{eqnarray*}
&&\int_\rn\int_{s}^T|b_{\ez_k}(r,X_{\ez_k}(s,r,x))-b(r,X(s,r,x))|\,dr\,d\mu\\
&&\quad\le
\int_\rn\int_{s}^T|b_{\ez_k}(r,X_{\ez_k}(s,r,x))-b(r,X_{\ez_k}(s,r,x))|\,dr\,d\mu\\
&&\quad\quad+
\int_\rn\int_{s}^T|b(r,X_{\ez_k}(s,r,x))-b(r,X(s,r,x))|\,dr\,d\mu=:I+II.
\end{eqnarray*}
By Theorem \ref{subexp-cr} and Lemma \ref{duality}, we see that
\begin{eqnarray*}
I&&\le \int_\rn\int_{s}^T|b_{\ez_k}(r,x)-b(r,x)|K_{s,r,\ez_k}(x)\,dr\,d\mu\\
&&\le \int_{s}^T
2\|b_{\ez_k}(r,\cdot)-b(r,\cdot)\|_{\Exp_\mu(L)}\|K_{s,r,\ez_k}\|_{L\log
L(\mu) }\,dr\to 0,\ \mbox{as}\ \  k\to\infty.
\end{eqnarray*}

On the other hand, by applying Lemma \ref{measure-convergence}, we find that
$$b(r,X_{\ez_k}(s,r,x))\to b(r,X(s,r,x))$$
a.e. in $(s,T)\times \rn$. Let $b_M:=\min\{\max\{b,-M\},M\}$. Notice that
\begin{eqnarray*}
\int_\rn\int_{s}^T|b_M(r,X(s,r,x))|\,dr\,d\mu
&&\le \int_\rn\int_{s}^T|b_M(r,X_{\ez_k}(s,r,x))-b_M(r,X(s,r,x))|\,dr\,d\mu\\
&&+ \int_\rn\int_{s}^T|b_M(r,X_{\ez_k}(s,r,x))|\,dr\,d\mu.
\end{eqnarray*}
By using the fact $X_{\ez_k}(s,t,x)$ converges to $X(s,t,x)$ a.e. on $[s,T]\times \rn$, we apply the dominated convergence theorem
and Theorem \ref{subexp-cr} to conclude that
\begin{eqnarray*}
\int_\rn\int_{s}^T|b_M(r,X(s,r,x))|\,dr\,d\mu&&\le \liminf_{k\to\infty}\int_\rn\int_{s}^T|b_M(r,X_{\ez_k}(s,r,x))|\,dr\,d\mu\\
&&\le  \liminf_{k\to\infty}\int_\rn\int_{s}^T|b(r,x)|K_{s,r,\ez_k}(x)\,dr\,d\mu\\
&&\le \liminf_{k\to\infty}\int_{s}^T\|b(r,\cdot)\|_{\Exp_\mu(\frac{L}{\log L})}\|K_{s,r,\ez_k}\|_{L\log L\log\log L(\mu)}\,dr\\
&&\le C(b)<\infty,
\end{eqnarray*}
where in the last second inequality we used Lemma \ref{duality} and the fact  $ L^{\Phi_\alpha}(\mu)\subset L\log L\log\log L(\mu)$
for any $\alpha>0$.  We therefore see that $b(r,X(s,r,x))\in L^1(s,T;\mu)$, and
\begin{eqnarray*}
II
&&\le \int_\rn\int_{s}^T|b(r,X_{\ez_k}(s,r,x))-b_M(r,X_{\ez_k}(s,r,x))|\,dr\,d\mu\\
&&+\int_\rn\int_{s}^T|b(r,X(s,r,x))-b_M(r,X(s,r,x))|\,dr\,d\mu\\
&&+ \int_\rn\int_{s}^T\left(|b_M(r,X_{\ez_k}(s,r,x))-b_M(r,X(s,r,x))|\right)\,dr\,d\mu\\
&&=:II_1+II_2+II_3.
\end{eqnarray*}
For each $\gz>0$, we can choose $M$ sufficient large such that $II_1+II_2<\gz/2$.
Applying the dominated convergence theorem to $II_3$, we see that
$$II_3\to 0\hspace{.5cm}\text{as }k\to \infty.$$
Hence, we obtain that
\begin{eqnarray*}
\lim_{k\to\infty}\int_\rn\int_{s}^T|b(r,X_{\ez_k}(s,r,x))-b(r,X(s,r,x))|\,dr\,d\mu=0,
\end{eqnarray*}
which together with the fact $X_{\ez_k}(s,t,x)\to X(s,t,x)$ a.e., implies that
$$X(s,t,x)=x+\int_{s}^tb(r,X(s,r,x))\,dr,\ \mu-a.e.$$
The proof is completed.
\end{proof}

Applying the above results of this section to the backward flow instead of the forward flow,
we can conclude that under the assumptions of the Main Theorem,
there exists a Borel map $\tilde X(s,t,x)$ arising as a limit of a
sequence of smooth flows $\tilde X_{\ez_k}(s,t,x)$, given as
\begin{equation}\label{inverse-flow-app}
\tilde X_{\ez_k}(s,t,x)=x-\int_{s}^tb_{\ez_k}(r,\tilde X_{\ez_k}(r,t,x))\,dr,
\end{equation}
such that  for each $s\in [0,t]$, it holds
\begin{equation}\label{inverse-flow}
\tilde X(s,t,x)=x-\int_{s}^tb(r,\tilde X(r,t,x))\,dr, \ a.e.\
x\in\rn.
\end{equation}

Then by using the fact that for $0\le s\le t\le T$ and each $\ez_k$,
$$X_{\ez_k}(s,t,\tilde X_{{\ez_k}}(s,t,x))=\tilde X_{\ez_k}(s,t,X_{{\ez_k}}(s,t,x))=x$$
(see \cite[Theorem  2.1]{cc05}), Lemma \ref{mc-flow} and Lemma \ref{measure-convergence}, we can
conclude that for $0\le s\le t\le T$, it holds
\begin{equation}\label{identity-flow}
X(s,t,\tilde X(s,t,x))=\tilde X(s,t,X(s,t,x))=x,
 \ a.e. \, x\in\rn.
\end{equation}

\begin{lem}\label{rep-density}
Let $X(s,t,x)$ be as in Lemma \ref{mc-flow} and $\tilde X(s,t,x)$ be
given as in \eqref{inverse-flow}.  Under the assumptions of the Main
Theorem,  for  $0\le s\le t\le T$, the density functions
$K_{s,t}=\frac{d}{d\mu} (X(s,t)_{\#}d\mu)$ and $\tilde K_{s,t}=\frac{d}{d\mu}
(\tilde X(s,t)_{\#}d\mu)$ satisfy
$$K_{s,t}(x)=\exp\left\{\int_s^t -\mathrm{div}_\mu\,b(r,\tilde X(r,t,x))\,dr\right\} \quad \ a.e. \ x\in \rn,$$
and
$$\tilde K_{s,t}(x)=\exp\left\{\int_s^t \mathrm{div}_\mu\,b(r, X(s,r,x))\,dr\right\} \quad \ a.e. \ x\in \rn.$$
\end{lem}
\begin{proof} We only give the proof for $K_{s,t}$ since the proof of $\tilde
K_{s,t}$ is the same. Notice that as we recalled in
\eqref{density-funct-smooth}, it holds for each $\ez_k$ that
$$K_{s,t,\ez_k}(x)=\exp\left\{\int_s^t -\mathrm{div}_\mu\,b_{\ez_k}(r,\tilde X_{\ez_k}(r,t,x))\,dr\right\},$$
where $\tilde X_{\ez_k}(r,t,x)$ is as in \eqref{inverse-flow-app}.
By Lemma \ref{measure-convergence},  we see that for  $0\le s\le t\le T$,
$$\mathrm{div}_\mu\, b(s,\tilde X_{\ez_k}(s,t,x))\to \mathrm{div}_\mu\, b(s,\tilde X(s,t,x))$$
in measure and a.e. $x\in \rn$ up to a subsequence of $\{\ez_k\}$.

Since $\mathrm{div}_\mu\, b(r,\tilde X_{\ez_k}(r,t,x))$ and
$\mathrm{div}_\mu\, b(r,\tilde X(r,t,x))$ are uniformly integrable in
$L^1(\mu)$, by the same argument in the proof of Lemma \ref{flow-eq}
we can further conclude that
$$\mathrm{div}_\mu\, b_{\ez_k}(r,\tilde X_{\ez_k}(r,t,x))\to \mathrm{div}_\mu\,  b(r,\tilde X(r,t,x))$$
in measure and a.e. $x\in \rn$ up to a subsequence of $\{\ez_k\}$.

From this together with the fact $ K_{s,t,\ez_k}(x)\rightharpoonup
K_{s,t}(x) \  \mathrm{in} \ L^{\Phi_\alpha}(\mu)$ from Lemma
\ref{density-est}, we can conclude that for $0\le s\le t\le T$ it
holds
$$K_{s,t}(x)=\exp\left\{\int_s^t -\mathrm{div}_\mu\,b(r,\tilde X(r,t,x))\,dr\right\} \quad \ a.e. \ x\in \rn,$$
as desired.
\end{proof}

Uniqueness of the flow will follow as a corollary of Theorem \ref{tran-main}.
\begin{prop}\label{uniq-flow}
  Under the assumptions of the Main Theorem,  the flows $X(s,t,x)$, $\tilde X(s,t,x)$ satisfying the
  properties from part (a) of the Main Theorem
  are unique.
  \end{prop}
  \begin{proof} Once more we only give the proof for $X(s,t,x)$ since the proof of $\tilde
X(s,t,x)$ is the same. By the well-posedness of the transport equation
(Theorem \ref{tran-main}), it suffices to show that for each $u_0\in
C^\infty_c(\rn)$, $u(s,x):=u_0(X(s,t,x))$ is a distributional solution
to the transport
  equation
  $$\dfrac{\partial}{\partial s}u(s,x)+b(s,x)\cdot\nabla u(s,x)=0$$
  on $(0,t)\times \rn$ with the final value $u(t)=u_0$.
That is, for each  $\varphi\in C^\infty((0,t]\times\rn)$ with compact support in $(0,t]\times \rn$, it holds
\begin{eqnarray*}
&&-\int_0^t\int_{\rn}u(s,x)\,\frac{\partial \varphi(s,x)}{\partial
s}\,ds\,d\mu(x)+ \int_{\rn}u(t,x)\, \varphi(t,x)\,d\mu(x)\\
&&\quad=\int_0^t\int_{\rn}\left[u(s,x)\vz\,\mathrm{div}_\mu(
b)(s,x)+u(s,x)b(s,x)\cdot\nabla \vz\right]\,d\mu(x)\,ds.
\end{eqnarray*}

From the fact $$K_{s,t}(x)=\frac{d}{d\mu}
(X(s,t, {x})_{\#}d\mu)=\exp\left\{\int_s^t -\mathrm{div}_\mu\,b(r,\tilde
X(r,t,x))\,dr\right\} \quad \ a.e. \ x\in \rn,$$ we know that for a.e.
$x\in \rn$ the density function $K_{s,t}(x)$ is absolutely continuous on
$[0,t]$. Using this, change of variables, integration by parts and
the fact
$$X(s,t,\tilde X(s,t,x))=\tilde X(s,t,X(s,t,x))=x,  \ a.e. \,
x\in\rn,$$
we obtain that
\begin{eqnarray*}
&&\int_0^t\int_{\rr^n} u\,\dfrac{\partial \varphi}{\partial s}\,d\mu\,ds=\int_0^t\int_{\rr^n} u_0(X(s,t,x))\,\dfrac{\partial \varphi(s,x)}{\partial
s}\,d\mu\,ds\\
&&\hs=\int_0^t\int_{\rr^n} u_0(y)\,\dfrac{\partial
\varphi(s,z)}{\partial s}\bigg|_{z=\tilde X(s,t,y)}K_{s,t}(y)\,d\mu\,ds\\
&&\hs=\int_0^t\int_{\rr^n} u_0(y)\left[\,\dfrac{\partial
\varphi(s,\tilde X(s,t,y))}{\partial s}-\nabla \varphi(s,\tilde
X(s,t,y))\cdot
b(s,\tilde X(s,t,y))\right]K_{s,t}(y)\,d\mu\,ds\\
&&\hs=\int_{\rr^n} u_0 \,\varphi(t,\cdot)\,d\mu-\int_0^t\int_{\rr^n} u_0(y) \varphi(s,
\tilde X(s,t,y))K_{s,t}(y)\mathrm{div}_\mu\,b(s,\tilde X(s,t,y))\,\,d\mu\,ds\\
&&\hs\hs   -\int_0^t\int_{\rr^n} u_0(y)\nabla \varphi(s,\tilde X(s,t,y))\cdot b(s,\tilde X(s,t,y))K_{s,t}(y)\,d\mu\,ds\\
&&\hs=\int_{\rr^n} u_0
\,\varphi(t,\cdot)\,d\mu-\int_0^t\int_{\rr^n} u_0(X(s,t,x))
\left[\varphi(s,x)\mathrm{div}_\mu\,b(s,x)+\nabla \varphi(s,x)\cdot
b(s,x)\right]\,\,d\mu\,ds.
\end{eqnarray*}
This implies $u_0(X(s,t,x))$ is a distributional solution to the
transport equation, as desired.
  \end{proof}

\begin{proof}[Proof of Main Theorem (a)]
The existence of  a forward flow $X$ and a backward flow $\tilde X$ follows
from Lemma \ref{flow-eq}. Property (i) follows from
\eqref{identity-flow}. The estimate of the density function follows
from Lemma \ref{density-est} and Lemma \ref{rep-density}. The
uniqueness follows from Proposition \ref{uniq-flow}.
\end{proof}

\section{Regularity, semigroup structure and stability}

\noindent
In this section, we prove part (b) of the Main Theorem, and  give a stability result. To do this, we start by stating the semigroup structure of our flow.

\begin{lem}\label{semigroup}
Let $b$ be as in the Main Theorem, and let $X$ and $\tilde X$ be the
forward and backward flows associated to $b$, respectively, that satisfy
properties of part (a) of the Main Theorem. Then $X$ and $\tilde X$
have the semigroup property.
\end{lem}
\begin{proof}
Let $0\le s\le t\le T$. By the proof of part (a) of the Main theorem from the last section, we know such $X(s,t,x)$ can be approximated by $X_{\ez_k}(s,t,x)$.
Notice that by the semigroup structure of
$X_{\ez_k}$, we have for each $0\le r\le s\le t\le T$ and a.e.
$x\in \rn$, it holds
$$X(r,t,x)=\lim_{k\to\infty}X_{\ez_k}(r,t,x)=\lim_{k\to\infty}X_{\ez_k}(s,t,X_{\ez_k}(r,s,x)), \ a.e.\,x\in\rn.$$
Therefore, to prove the semigroup structure, it suffices to show
that
$$\lim_{k\to\infty}X_{\ez_k}(s,t,X_{\ez_k}(r,s,x))=X(s,t, X(r,s,x)),\ a.e.\,x\in\rn.$$
Write
\begin{eqnarray*}
&&\left|X_{\ez_k}(s,t,X_{\ez_k}(r,s,x))-X(s,t,X(r,s,x))\right|\\
&&\le \left|X_{\ez_k}(s,t,X_{\ez_k}(r,s,x))-X(s,t,
X_{\ez_k}(r,s,x))\right|+ \left|X(s,t,
X(r,s,x))-X(s,t,X_{\ez_k}(r,s,x))\right|=:I+II.
\end{eqnarray*}
By Lemma \ref{mc-flow}, we see that $X_{\ez_k}(s,t, \cdot)$
converges to $X(s,t,\cdot)$ in measure. Let $c>0$ be fixed. Then for
each $\gz>0$, there exists $k_\gz$, such that for $k>k_\gz$, it
holds
$$\mu\left(\{x:\,|X_{\ez_k}(s,t,x)-X(s,t,x)|>c\}\right)<\gz.$$
Let $E_{k,c}=\{x:\,|X_{\ez_k}(s,t,x)-X(s,t,x)|>c\}$. Recall that by
{Lemma \ref{measure-control}}, for any measurable set $E$ with sufficient
small measure, it holds
$$
\left|\log\log\log\left(
\frac{1}{\int_\rn\chi_E(X_{\ez_k}(r,s,x))\,d\mu}\right)-
\log\log\log\left(\frac{1}{\mu(E)}\right)\right|\leq
C\int_r^s\beta(h)\,dh,
$$
since $\div_\mu b_{\ez_k}$ has uniform bound in $\Exp_\mu(\frac{L}{\log L})$.
We then can conclude that
\begin{eqnarray*}
\mu\left(\{x:\,|X_{\ez_k}(s,t,X_{\ez_k}(r,s,x))-X(s,t,
X_{\ez_k}(r,s,x))|>c\}\right)
&&=\int_{\rn}\chi_{E_{k,c}}(X_{\ez_k}(r,s,x))\,d\mu\\
&&\le \exp\left\{-\left(\log\left(\frac{1}{\mu(E_{k,c})}\right)\right)^{\exp\{-C\int_r^s\beta(h)\,dh\}}\right\}\\
&&\le \exp\left\{-\left(\log\left(\frac{1}{\gz}\right)\right)^{\exp\{-C\int_0^T\beta(h)\,dh\}}\right\},
\end{eqnarray*}
which implies that
$$\lim_{k\to\infty}\mu\left(\{x:\,|X_{\ez_k}(s,t,X_{\ez_k}(r,s,x))-X(s,t,
X_{\ez_k}(r,s,x))|>c\}\right)=0.$$ On the other hand, using Lemma
\ref{measure-convergence}, we see that
$X(s,t,X_{\ez_k}(r,s,x))$ converges to $X(s,t,X(r,s,x))$ in
measure. Therefore, we see that $X_{\ez_k}(s,t,X_{\ez_k}(r,s,x))$
converges in measure to $X(s,t,X(r,s,x))$, up to a subsequence, we
can conclude that
$$X(s,t,X(r,s,x))=\lim_{k\to\infty}X_{\ez_k}(s,t,X_{\ez_k}(r,s,x)),\ a.e.\, x\in\rn.$$
The same argument works for $\tilde X$, and the proof is completed.
\end{proof}

\noindent We are now in position to complete the proof of our Main
Theorem.

\begin{proof}[Proof of  Main Theorem (b)]
We already know that a flow $X$ associated to $b$ satisfying
properties of part (a) exists and is unique. Further, by Lemma
\ref{semigroup} we also know it has semigroup structure. Thus, in
order to prove that $X$ is a regular flow it just remains to show
that ${X(s,t,\cdot)_{\#}\,dx\ll\,dx}$. For each $\psi\in
C^\infty_c(\rn)$, we have
\begin{eqnarray*}
\int_\rn |\psi(X(s,t,x))|\,dx&&=\int_\rn (2\pi)^{n/2}|\psi(X(s,t,x))|\exp\left\{\frac{|x|^2}{2}\right\}\,d\mu(x)\\
&&=\int_\rn (2\pi)^{n/2}|\psi(y)|\exp\left\{\frac{|\tilde X(s,t,y)|^2}{2}\right\}K_{s,t}(y)\,d\mu(y),
\end{eqnarray*}
where $\tilde X(s,t,y)$
is the inverse map of $X(s,t,y)$ as indicated in \eqref{identity-flow}.
From the assumption $$\frac{|b|}{1+|x|\log^+|x|}\in L^1(0,T;L^\infty),$$
we can see that
$\{\,\tilde X(s,t,y):\,y\in \supp\psi\}$ is bounded in $[s,T]\times \rn$ . Therefore,
\begin{eqnarray*}
\int_\rn |\psi(X(s,t,x))|\,dx&&\le C(b,s,t,\psi)
\int_\rn |\psi(y)|K_{s,t}(y)\,dy,
\end{eqnarray*}
and hence, $X(s,t,\cdot)_{\#}\,dx\ll\,dx$. Apparently, the above arguments apply to $\tilde X$ and the same conclusion holds.
The proof is completed.
\end{proof}

As a result of the techniques we have used throughout this work
we get the following result about stability.

\begin{thm}\label{stability-dpl}
Let $b,\{b_k\}\in L^1(0,T;W^{1,1}_{\mathrm{loc}})$ satisfying
$$\frac{|b(t,x)|}{1+|x|\log^+|x|}, \frac{|b_k(t,x)|}{1+|x|\log^+|x|} \in L^1(0,T;L^\infty)$$
and
$$b_k\to b \quad \mbox{in} \ \Exp_\mu(L). $$
Assume that $\div_\mu\,b,\div_\mu\,b_k$ are uniformly bounded in
$L^1(0,T;\Exp_\mu(\frac{L}{\log L}))$ and $\div_\mu\,b_k$ converges
to $\div_\mu \,b$ in $L^1(0,T;L^1_\loc(\mu))$ . Let
$X(s,t,x),\{X_k(s,t,x)\}$, that satisfy properties from part (a) of the
Main Theorem, be the forward (or backward) flows generated from $b,\{b_k\}$ respectively.
Then
\begin{eqnarray}
\lim_{k\to\infty}\int_\rn\sup_{t\in [s,T]} \left|X(s,t,x)-X_k(s,t,x)\right|\,d\mu\to 0.
\end{eqnarray}
\end{thm}
\begin{proof}
For each bounded function $\beta\in C^1(\rr,\rr)$, $\beta(X^i_k(s,t,x))$, $1\le i\le n$, is the solution to the
Cauchy problem of the transport equation associated to the vector field $b_k$, with the final value  $\beta(x^i)$.
By weak-$\ast$ compactness in $L^\infty(s,T;L^\infty)$,
we see that there exists a subsequence $\{\beta(X^i_{k_j})\}_j$  converging to a function $\widetilde X$, which is a solution
 to the
Cauchy problem of the transport equation associated to the vector field $b$, with the initial value  $\beta(x^i)$.
By the uniqueness, we get that $\widetilde X=\beta(X^i(s,t,x))$.

By the well-posedness and the renormalization property of the transport equation, we deduce that,  indeed, $\beta(X^i_{k})$
converges in measure to $\beta(X(s,t,x))$.
Following the same argument  as in Lemma \ref{mc-flow}, we see that
$X_{k}$ converges in measure to $X(s,t,x)$.

Observing this, and the fact
\begin{eqnarray*}
&&\int_\rn\sup_{t\in [s,T]} \left|X(s,t,x)-X_k(s,t,x)\right|\,d\mu\\
&&\quad\le \int_\rn\int_s^T|b(r,X(s,r,x))-b_k(r,X_k(s,r,x))|\,dr\,d\mu\\
&&\quad\le \int_\rn \int_s^T|b(r,X(s,r,x))-b(r,X_k(s,r,x))|\,dr\,d\mu\\
&&\quad\quad+\int_\rn\int_{{s}}^T|b(r,X_k(s,r,x))-b_k(r,X_k(s,r,x))|\,dr\,d\mu,
\end{eqnarray*}
we can follow the proof of Lemma \ref{flow-eq} to conclude that
\begin{eqnarray*}
\lim_{k\to\infty}\int_\rn\sup_{t\in [s,T]} \left|X(s,t,x)-X_k(s,t,x)\right|\,d\mu\to 0.
\end{eqnarray*}
The proof is completed.
\end{proof}

\begin{rem}\label{manylogs}\rm
From the proofs of the paper and our previous paper \cite{cjmo}, it
looks that one can strengthen the borderline condition on the
divergence of vector fields $b$ a bit more as
$$\div_\mu b\in  L^1\left(0,T;\Exp_\mu\left(\frac{L}{\log L\,\log\log L\,\dots\,\underbrace{\log\cdots\log}_{k}L} \right)\right)$$
in order to get well-posedness of the ODE. However, since this would
require much more tedious calculations, we will not go through it
here.
\end{rem}

\begin{rem}\rm
It will be interesting to know if one can adopt recent developments
of regular Lagrangian flows (cf. \cite{ac14}) and use the continuity
equation rather than the transport equation, to improve the Main
theorem.
\end{rem}

\subsection*{Acknowledgment}
\hskip\parindent {We thank the anonymous referee for their helpful comments and
suggestions, which significantly contributed to improve the quality of this paper.}
Albert Clop, Joan Mateu and Joan Orobitg were
partially supported by Generalitat de Catalunya (2014SGR75) and
Ministerio de Econom\'\i a y Competitividad (MTM2013-44699). Albert
Clop was partially supported by the Programa Ram\'on y Cajal (Spain). Renjin
Jiang was partially supported by National Natural Science Foundation
of China (NSFC 11301029). All authors were partially supported by
Marie Curie Initial Training Network MAnET (FP7-607647).

\noindent
\textit{Albert Clop, Joan Mateu and Joan Orobitg}

\vspace{0.1cm}
\noindent
Departament de Matem\`atiques, Facultat de Ci\`encies,\\
Universitat Aut\`onoma de Barcelona\\
08193 Bellaterra (Barcelona), CATALONIA.

\vspace{0.3cm}

\noindent
\textit{Renjin Jiang}

\vspace{0.1cm}
\noindent
School of Mathematical Sciences, Beijing Normal University, Laboratory of Mathematics and Complex Systems,
Ministry of Education, 100875, Beijing, CHINA

and

\noindent
Departament de Matem\`atiques, Facultat de Ci\`encies,\\
Universitat Aut\`onoma de Barcelona\\
08193 Bellaterra (Barcelona), CATALONIA.

\vspace{0.2cm}
\noindent{\it E-mail addresses}:\\
\texttt{albertcp@mat.uab.cat}\\
\texttt{rejiang@bnu.edu.cn}\\
\texttt{mateu@mat.uab.cat}\\
\texttt{orobitg@mat.uab.cat}

\begin{thebibliography}{999}

\vspace{-0.3cm}
\bibitem{a04} Ambrosio L., Transport equation and Cauchy problem for BV vector fields,
Invent. Math. 158 (2004), 227-260.

\vspace{-0.3cm}
\bibitem{am08} Ambrosio L., Transport equation and Cauchy problem for
non-smooth vector fields, Calculus of variations and nonlinear
partial differential equations, 1-41, Lecture Notes in Math., 1927,
Springer, Berlin, 2008.



\vspace{-0.3cm}
\bibitem{acf14} Ambrosio L., Colombo M., Figalli A.,
Existence and uniqueness of maximal regular flows for non-smooth vector fields,
 Arch. Ration. Mech. Anal. 218 (2015), 1043-1081.

\vspace{-0.3cm}
\bibitem{ac14}
 Ambrosio L., Crippa G., Continuity equations and ODE flows with
 non-smooth velocity, Proc. Roy. Soc. Edinburgh Sect. A 144 (2014), 1191-1244.

\vspace{-0.3cm}
\bibitem{acfs09} Ambrosio L., Crippa  G., Figalli A., Spinolo L.V.,
Some new well-posedness results for continuity and transport
equations, and applications to the chromatography system, SIAM J.
Math. Anal. 41 (2009), 1890-1920.


\vspace{-0.3cm}
\bibitem{af09} Ambrosio L., Figalli A., On flows associated to Sobolev vector fields in
Wiener spaces: an approach \`a la DiPerna-Lions, J. Funct. Anal. 256 (2009), 179-214.



\vspace{-0.3cm}
\bibitem{bc13} Bouchut F., Crippa G., Lagrangian flows for vector fields with
gradient given by a singular integral, J. Hyperbolic Differ. Equ. 10 (2013), 235-282.

\vspace{-0.3cm}
\bibitem{bo98}  Bogachev V., Gaussian measures. Mathematical
Surveys and Monographs, 62. American Mathematical Society, Providence, RI, 1998.

\vspace{-0.3cm}
\bibitem{cc05} Cipriano F., Cruzeiro A.B.,
Flows associated with irregular $\rr^d$-vector fields,
J. Differential Equations 219 (2005), 183-201.

\vspace{-0.3cm}
\bibitem{cjmo} Clop A., Jiang R., Mateu J., Orobitg J., Linear transport equation
for vector fields with subexponentially integrable divergece, to
appear in  Calc. Var. Partial Differential Equations (arXiv:1502.05303).

\vspace{-0.3cm}
\bibitem{ccr06} Colombini F., Crippa G., Rauch J., A note on two-dimensional transport
with bounded divergence, Comm. Partial Differential Equations 31
(2006), 1109-1115.

\vspace{-0.3cm}
\bibitem{cl02} Colombini F., Lerner N., Uniqueness of continuous solutions for
BV vector fields, Duke Math. J. 111 (2002), 357-384.

\vspace{-0.3cm}
\bibitem{cr09} Crippa G., The flow associated to weakly differentiable vector fields. Tesi. Scuola Normale Superiore di Pisa (Nuova Series) [Theses
    of Scuola Normale Superiore di Pisa (New Series)], 12. Edizioni della Normale, Pisa, 2009. xvi+167 pp.

\vspace{-0.3cm}
\bibitem{cdl08} Crippa G., De Lellis C., Estimates and regularity results for the DiPerna-Lions flow,
J. Reine Angew. Math. 616 (2008), 15-46.


\vspace{-0.3cm}
\bibitem{cr83}
Cruzeiro A.B., \'Equations diff\'erentielles ordinaires: non explosion et mesures quasi-invariantes,
J. Funct. Anal. 54 (1983), 193-205.

\vspace{-0.3cm}
\bibitem{d96}
Desjardins B., A few remarks on ordinary differential equations, Comm.
Partial Diff. Eq. 21 (1996), 1667-1703.

\vspace{-0.3cm}
\bibitem{dl89} DiPerna R.J., Lions P.L., Ordinary differential equations, transport theory
and Sobolev spaces, Invent. Math. 98 (1989), 511-547.




\vspace{-0.3cm}
\bibitem{ha80} Hale J.K.,
Ordinary differential equations. Second edition. Robert E. Krieger Publishing Co., Inc.,
Huntington, N.Y., 1980.


\vspace{-0.3cm}
\bibitem{rr91} Rao M., Ren Z., Theory of Orlicz spaces, Dekker, New York, 1991.




\end{thebibliography}
\end{document}